\documentclass[12pt]{article}
\usepackage[colorlinks,
            linkcolor=blue,
            anchorcolor=green,
            citecolor=blue
            ]{hyperref}
\usepackage{mathrsfs,amssymb,amsfonts}
\usepackage{amsmath,amssymb,latexsym,color}
\usepackage[mathscr]{eucal}
\usepackage[numbers,sort&compress]{natbib}
\usepackage{graphicx}
\usepackage{lipsum}
\usepackage{enumerate}
\def\thefootnote{\fnsymbol{footnote}}
\newtheorem{thm}{Theorem}[section]

\newtheorem{lemma}[thm]{Lemma}

\newtheorem{example}[thm]{Example}

\newtheorem{problem}[thm]{Problem}

\newtheorem{remark}[thm]{Remark}
\newtheorem{Notation}[thm]{Notation}

\newtheorem{con}[thm]{Conjecture}

\newcommand{\proof}{{\it Proof.\quad}}
\newcommand{\qed}{\hfill\Box\medskip}
\renewcommand{\thefootnote}{\arabic{footnote}}

\newcommand{\di}{{\rm d}}
\newcommand{\dm}{{\rm dim_m}}

\begin{document}
\renewcommand{\abovewithdelims}[2]{
\genfrac{[}{]}{0pt}{}{#1}{#2}}

\title{\bf On the mixed metric dimension of $2$-connected graphs
}

\author{{Shi Chen and Xuanlong Ma\footnote{Corresponding author}}
\medskip
\\
{\small\em School of Science, Xi'an Petroleum University, Xi'an 710065, P. R. China}
\medskip
}

 \date{}
 \maketitle
\newcommand\blfootnote[1]{%
\begingroup
\renewcommand\thefootnote{}\footnote{#1}%
\addtocounter{footnote}{-1}%
\endgroup
}
\begin{abstract}
We show that for a $2$-connected graph $G$ which is not a cycle, the mixed metric dimension of $G$ is at most $2c(G)$, where $c(G)$ is the cyclomatic number of $G$.
As an immediate application, we prove a conjecture
proposed by Sedlar and \v{S}krekovski (2021).

\medskip
\noindent {\em Keywords:} Mixed metric dimension; Cyclomatic number; $2$-connected graph

\medskip
\noindent {\em 2020 MSC:} 05C12; 05C38
\end{abstract}

\blfootnote{{\em E-mail addresses:} nuoyoushi@gmail.com (Chen), xuanlma@xsyu.edu.cn (Ma)}

\section{Introduction}

Every graph considered in our paper is finite, simple, and connected.
Let $G$ be a graph. Denote by $V(G)$ and $E(G)$ the vertex set and edge set of $G$, respectively.
Let $u$ and $v$ be two vertices of $G$. The distance between $u$ and $v$ is denoted by $\di_G(u,v)$, or simply by $\di(u,v)$ when no confusion can arise.
For distinct $x,y\in V(G)$, we say a vertex $v\in V(G)$
{\em distinguishes} $x$ and $y$ whenever
$
\di(v,x)\ne \di(v,y).
$
A subset $S \subseteq V(G)$ is called a {\em metric generator} of $G$ if for every
pair of distinct vertices $x,y\in V(G)$, there exists a vertex in $S$ such that this vertex distinguishes $x$ and $y$.
A metric generator of $G$ with minimum cardinality is called
a {\em metric basis}
of $G$, and the cardinality of a metric basis is referred to as the {\em metric dimension} of $G$.
The concept of metric dimension was independently introduced in the 1970s by Harary and Melter \cite{Ham} and Slater \cite{Sl}. The aim is to distinguish vertices of a graph using distances.
The metric dimension of a graph has a wide range of applications, including robot navigation in networks \cite{Kh}, chemistry \cite{CPZ}, pattern recognition and image processing \cite{Me}, and the game of Mastermind \cite{CA}.

As a further variation of the classical notion of a metric generator, the concept of an edge metric generator was introduced in \cite{KT} to distinguish the edges of a graph by means of distances from a prescribed set of vertices. Although a metric basis of a connected graph $G$ uniquely determines every vertex of $G$ through its distance representation, it does not necessarily distinguish all edges of $G$. This observation motivates the following definitions.
For a vertex $v\in V(G)$ and an edge $e=\{u,w\}\in E(G)$, the distance between $v$ and $e$ is defined by
$
\di(v,e)=\min\{\di(v,u),\di(v,w)\}.
$
A vertex $x\in V(G)$ is said to distinguish two distinct edges $e_1,e_2\in E(G)$ on condition that
$
\di(x,e_1)\ne \di(x,e_2).
$
A set $S\subseteq V(G)$ is called an {\em edge metric generator} for $G$ provided that every pair of distinct edges of $G$ is distinguished by some vertex of $S$. An edge metric generator of minimum cardinality is called an {\em edge metric basis} of $G$, and its cardinality is called the {\em edge metric dimension} of $G$.

Given a graph $G$, we write
$$
O(G)=V(G)\cup E(G).
$$
A vertex $v\in V(G)$ is said to distinguish two distinct elements $x,y\in O(G)$ (vertices or edges) if
$
\di(v,x)\ne \di(v,y).
$
A set $S\subseteq V(G)$ is called a {\em mixed metric generator} for $G$ if for every pair of distinct elements $x,y\in O(G)$, there exists a vertex of $S$ such that
it can distinguish $x,y$.
The cardinality of a mixed metric generator with minimum size, denoted by $\dm(G)$, is called the {\em mixed metric dimension} of $G$.
The mixed metric dimension of a graph was first introduced in \cite{Kek}.
For a comprehensive overview of the variants of metric dimension and their applications to network detection, graph isomorphism, coin-weighing problems, and machine learning, we refer the reader to the two surveys \cite{Ku,Ti}.

In \cite{Kek},
the authors determined the mixed metric dimension of several families of graphs, including cycles, complete bipartite graphs, trees, and grid graphs.
They also showed that the problem of determining the
mixed metric dimension of a graph is NP-hard in the general case.
A {\em leaf} of $G$ is a vertex with degree $1$.
Equivalently, a leaf of $G$ is a vertex of $G$ that is incident to exactly one edge. Denote by $\ell(G)$
the number of all leaves of $G$.
The {\em cyclomatic number} of $G$, denoted by $c(G)$, is
the number of independent cycles of $G$. Namely $$c(G)=|E(G)|-|V(G)|+1.$$
In \cite{Sed21},
Sedlar and \v{S}krekovski obtained the exact value of the mixed metric dimension of a unicyclic graph and they proposed the following conjecture:

\begin{con}\label{conj}{\rm (\cite[Conjecture 13]{Sed21})}
Let $G$ be a connected graph which is not a cycle.
Then $\dm(G)\le \ell(G)+2c(G)$.
\end{con}

Conjecture~\ref{conj} holds for trees and $3$-connected graphs, see \cite{Sed21}.
In \cite{Sed}, the authors showed that Conjecture~\ref{conj} holds for all Theta graphs. Recently, Chakraborty, Foucaud, and  Hakanen \cite{Ch} proved that Conjecture~\ref{conj} holds for a connected graph which is not a tree such that $\ell(G)\ge 1$, and a connected graph $G$ with $\ell(G)=0$ and has a cut-vertex.
In \cite{Hua}, the authors gave some graphs satisfying  Conjecture~\ref{conj}. 
In this paper, we will prove the following theorem:

\begin{thm}\label{other}
Let $G$ be a $2$-connected graph which is not a cycle.
Then $$\dm(G)\le 2c(G).$$
\end{thm}
As an immediate application of Theorem~\ref{other}, we confirm Conjecture~\ref{conj}.

\section{Preliminaries}
Let $G$ be a graph.
A {\em thread} in a graph $G$ is a path whose internal vertices all have degree $2$ in $G$.
Let $P=(p_0,p_1,\ldots,p_l)$ be a thread of $G$.
We use $G-P^0$ to denote the subgraph of $G$ that is obtained by removing the internal vertices of this thread $P$ and keeping the endpoints $p_0$ and $p_l$.
For convenience, throughout this paper, we denote the vertex $p_i$ of $P$ by $[i,i]$, and the edge $\{p_{j}, p_{j+1}\}$ of $P$ by $[j,j+1]$, where $0\le i \le l$ and $0\le j \le l-1$.
Namely, every element of $P$ can be expressed as
$$
[a,b],~~b-a\in \{0,1\},~~0\le a \le b \le l.
$$
If $G-P^0$ is connected and $l<\di_{G-P^0}(p_0,p_l)$, then
$P$ is called a {\em short thread} of $G$. If $\deg(p_0)\ge 3$ and $\deg(p_l)\ge 3$, then this thread $P$ is called
{\em maximal}. By \cite{Wh}, a $2$-connected graph which is not a cycle can be decomposed into distinct maximal threads.
Write
$$
\mathcal{B}=\{v\in V(G): \deg(v)\ge 3\}.
$$

\begin{lemma}\label{lem-1}
Let $G$ be a graph with a $u,v$-thread $P$ of length $l$. If there exists $z\notin V(P)$ such that
$\di(u,z)+\di(v,z)\le l$, then every shortest path from
$u$ to $z$ does not pass through any internal vertex of
$P$.
\end{lemma}
\proof
Suppose, for a contradiction, that there exists a shortest path from $u$ to $z$ pass through some internal vertex of
$P$. Then clearly, $\di(u,z)=l+\di(v,z)$. It follows that
$l+2\di(v,z)=\di(u,z)+\di(v,z)\le l$, which is impossible as $\di(v,z)\ne 0$.
$\qed$

\begin{lemma}\label{lem-2}
Let $G$ be a graph with a $u,v$-thread $P$ of length $l$. If there exists an edge $e=\{a,b\}\in E(G)\setminus E(P)$ such that
$\di(u,e)+\di(v,e)\le l-1$, then there exists a $u,v$-path $W$ in $G-P^0$ such that $|W|\le l$.
\end{lemma}
\proof
Assume that $\di(u,e)=\di(u,u')$ and $\di(v,e)=\di(v,v')$, where $u',v'\in \{a,b\}$.

Suppose, for a contradiction, that there exists a shortest path from $u$ to $u'$ pass through some internal vertex of
$P$. Then we have that $\di(u,u')=l+\di(v,u')$.
Note that $\di(v,u')\ge \di(v,e)$.
It follows that
$$\di(u,e)+\di(v,e)=\di(u,u')+\di(v,v')=l+\di(v,u')+\di(v,v')\le l-1,$$
which is a contradiction. We conclude that every shortest path from $u$ to $u'$ do not pass through any internal vertex of
$P$. Similarly, every shortest path from $v$ to $v'$ do not pass through any internal vertex of $P$.

Take now a shortest path from $u$ to $u'$ and a shortest path from $v$ to $u'$. If $u'=v'$,
then we concatenate the two paths at
$u'$ to obtain a $u,v$-path, say $W$, which implies that
$W$ in $G-P^0$ and $|W|=\di(u,e)+\di(v,e)\le l-1< l$, as desired.
If $u'\ne v'$, then we obtain the $u,v$-path of $G-P^0$ has length $\di(u,e)+\di(v,e)+1\le l$, as wanted.
$\qed$

\begin{lemma}\label{lem-3}
Let $G$ be a $2$-connected graph without leaves.
Then $|\mathcal{B}|\le 2c(G)-2$.
\end{lemma}
\proof
It is well known that $\sum_{v\in V(G)}\deg(v)=2|E(G)|$.
Note that $G$ has no leaves.
We deduce that
$$
\sum_{v\in V(G)}(\deg(v)-2)=\sum_{v\in V(G)}\deg(v)-2|V(G)|
=2(|E(G)|-|V(G)|)=2c(G)-2.
$$
It follows that $|\mathcal{B}|\le \sum_{v\in V(G)}(\deg(v)-2)=2c(G)-2,$ as desired.
$\qed$

\begin{lemma}\label{lem-4}
Let $G$ be a graph with a thread $P=(p_0,p_1,\ldots,p_l)$, where $p_0=u$ and $p_l=v$. Suppose that $H=G-P^0$ is connected and $l\ge \di_H(u,v)$. Then
$$\dm(G)\le \dm(H)+2.$$
In particular, if $\mathcal{R}$ is a mixed metric generator of $H$, then $\mathcal{R}\cup \{u,w\}$ is a mixed metric generator of $G$, where $w=p_{\lceil l/2\rceil}$.
\end{lemma}
\proof
Write $d=\di_H(u,v)$.
We first claim that for any pair of vertices $a,b\in V(H)$,
$\di_H(a,b)=\di_G(a,b)$. Indeed, $\di_G(a,b)\le \di_H(a,b)$.
In $G$, if there exists a shortest path from $a$ to $b$
that avoids the internal vertices of $P$, then
$\di_G(a,b)=\di_H(a,b)$, as desired. Otherwise, every  shortest path from $a$ to $b$ must contain an internal vertex of $P$. Since $P$ is a thread, it follows that every shortest path from $a$ to $b$ must pass through all internal vertices of $P$, which implies any shortest path from $a$ to $b$ contains the entire $P$. Now, notice that $H$ is connected and $l\ge d$. It follows that $\di_G(a,b)\ge \di_H(a,b)$, and so we have $\di_G(a,b)=\di_H(a,b)$. We conclude that the above claim is valid.

Now, let $\mathcal{R}$ be a mixed metric generator of $H$ with   the smallest cardinality, that is $|\mathcal{R}|=\dm(H)$.
Without loss of generality, we may assume that there exists
$r\in \mathcal{R}$ such that $\di_H(r,u)\ge \di_H(r,v)$.
Write
$$A=\di_H(r,u),~~B=\di_H(r,v),~~\delta_r=A-B,$$
then $\delta_r\ge 0$.
Furthermore,  any element of $P$ can be represented as
$$
z=[a,b],~~b-a\in \{0,1\},~~0\le a \le b \le l.
$$
Note that $d=\di_H(u,v)$. It follows that
\begin{equation}\label{eq-1}
\di_G(u,z)=\min\{a,d+l-b\}
\end{equation}
and
\begin{equation}\label{eq-2}
\di_G(r,z)=\min\{A+a,B+l-b\}=B+\min\{\delta_r+a,l-b\}.
\end{equation}
Now let $m=\lceil \frac{l}{2}\rceil$ and let $w=p_m$. Then
\begin{equation}\label{eq-3}
\di_G(w,z)=\left\{
                                  \begin{array}{ll}
                                    m-b, & \hbox{if
                                    $b\le m$;}\\
                                    a-m, & \hbox{if $a\ge m$.}
                                  \end{array}
                                \right.
\end{equation}

{\bf Claim 1.} Any two distinct elements of $P$ are distinguished by some vertex of $\{w,u,r\}$.

We say that an element $[a,b]$ of $P$ lies on the left (resp. right) side of $w$ if $b\le m$ (resp. $a\ge m$).
Naturally, the element $[m,m]$ of $P$ lies on both the left and the right sides of $w$.
In the following, we consider two cases.

\medskip
\noindent {\bf Case 1.} The two distinct elements of $P$ lie on the same side of $w$.
\medskip

If the two distinct elements of $P$ lie on the left side of $w$, then by \eqref{eq-1} and \eqref{eq-3}, we see that the two distinct elements are distinguished by some vertex of $\{u,w\}$, as desired. Suppose now that the two distinct elements of $P$ lie on the right side of $w$, say $[a,b]$ and $[x,y]$. If $a\ne x$, then $w$ can distinguish them by \eqref{eq-3}. In the following, assume $a=x$. Then by \eqref{eq-2}, we have that
\begin{equation}\label{eq-4}
\di_G(r,[a,b])=B+\min\{\delta_r+a,l-b\},~~
\di_G(r,[x,y])=B+\min\{\delta_r+a,l-y\}.
\end{equation}
Since $y\ge a$, $b\ge a$, and $2a\ge l$, we have that $a\ge l-a\ge l-b$ and $a\ge l-y$. Note that $b\ne y$. It follows from \eqref{eq-4} that $[a,b]$ and $[x,y]$ are distinguished by $r$, as desired.

\medskip
\noindent {\bf Case 2.} The two distinct elements $[a,b]$ and $[x,y]$ of $P$ lie on opposite sides side of $w$.
\medskip

Without loss of generality, assume that $z_1=[a,b]$ lies on the left side of $w$ and $z_2=[x,y]$ lies on the right side of $w$.
Namely, $b\le m \le x$. From \eqref{eq-1}, it follows that
$\di_G(u,z_1)=\min\{a,d+l-b\}=a$ and $\di_G(u,z_2)=\min\{x,d+l-y\}$. If $a\ne \min\{x,d+l-y\}$, then $z_1$ and $z_2$ are distinguished by $u$, as desired. Thus, in the following, we assume that
\begin{equation}\label{eq-5}
a=\min\{x,d+l-y\}.
\end{equation}

Suppose that $a=x$. Then $z_1=[m,m]=w$ and $z_2=[m,m+1]$.
We deduce that $z_1$ and $z_2$ lie on the right side of $w$, and so Case 1 ends the proof.

We next suppose that $a<x$. If $b=y$, then $y=b\le m \le x \le y$, which implies that $z_1$ and $z_2$ lie on the left side of $w$, and so Case 1 completes our proof.  Thus, in the following, assume that $b\ne y$, that is, $y>b$.
Now, in view of \eqref{eq-2}, we have that
$$
\di_G(r,z_1)=B+\min\{\delta_r+a,l-b\},~~
\di_G(r,z_2)=B+\min\{\delta_r+x,l-y\}.
$$
It follows from \eqref{eq-5} that $a=d+l-y$. Since $d\ge 1$ and $\delta_r\ge 0$, we conclude that
$$l-y=a-d<a<x\le x+\delta_r.$$
As a result, we have $\di_G(r,z_2)=B+l-y$.
If $\di_G(r,z_1)=B+l-b$, then clearly, $\di_G(r,z_1)\ne \di_G(r,z_2)$, and so $z_1$ and $z_2$ are distinguished by $r$, as wanted. Otherwise, $\di_G(r,z_1)=B+\delta_r+a$. Consequently, we have
$$
B+\delta_r+a-(B+l-y)=\delta_r+a-(a-d)=\delta_r+d>0.
$$
This forces that $\di_G(r,z_1)\ne \di_G(r,z_2)$, which implies that $z_1$ and $z_2$ are distinguished by $r$, as required. Thus, Claim $1$ is valid.

{\bf Claim 2.} For any two elements $z\in O(P)\setminus O(H)$ and $z'\in O(H)$, there exists an element in $\{u,w\}\cup \mathcal{R}$ such that it can distinguish $z$ and $z'$.

Note that $P$ is a thread. It follows that
\begin{equation}\label{eq-6}
\di_G(w,z')=\min\{m+\di_H(u,z'),l-m+\di_H(v,z')\}\ge l-m.
\end{equation}
If $l=2k$ for some positive integer $k$, then
$\di_G(w,z)\le k-1$ and $\di_G(w,z')\ge k$, which implies that
$w$ distinguishes $z$ and $z'$, as required.

In the following, we assume that $l=2k+1\ge 3$ for some positive integer $k$, that is, $m=k+1$. Then by \eqref{eq-3} and \eqref{eq-6}, it is easy to see that if $z\notin \{[1,1],[0,1]\}$, then $z$ and $z'$ are distinguished by $w$.
In the following, we consider two cases.

\medskip
\noindent {\bf Case 1.} $z=[1,1]$.
\medskip

Indeed, we may assume that $w$ does not distinguish $z$ and $z'$. It follows that
$$
\di_G(w,z')=\di_G(w,z)=k.
$$
In view of \eqref{eq-6}, it can be inferred that $l-m+\di_H(v,z')=k+\di_H(v,z')=k$, and hence $\di_H(v,z')=0$.
Now, by \eqref{eq-2}, we have that
$$
\di_G(r,z)=B+\min\{\delta_r+1,l-1\}\ge B+1.
$$
However, because of $\di_H(v,z')=0$, we have
$\di_G(r,z')\le B$. It follows that $r$ distinguishes $z$ and $z'$, as desired.

\medskip
\noindent {\bf Case 2.} $z=[0,1]$.
\medskip

Clearly, we may assume that both $w$ and $u$ do not distinguish $z$ and $z'$. Therefore, we have
$\di_G(u,z')=\di_G(u,z)=0$, which means
$\di_G(u,z')=\di_H(u,z')=0$. Moreover, we also have
$\di_G(w,z')=\di_G(w,z)=k$ by \eqref{eq-3}.
Now
$$
\di_G(w,z')=\min\{k+1+\di_H(u,z'),k+\di_H(v,z')\}=k=k+\di_H(v,z')
$$
implies $\di_H(v,z')=0$. As a consequence,
$\di_H(v,z')=\di_H(u,z')=0$. We conclude that $z'=\{u,v\}\in E(G)$, that is, $\di_G(u,v)=\di_H(u,v)=1$. It follows that there exists a vertex $q\in \mathcal{R}$ such that $q$ can distinguish $u$ and $z'$. Consequently,
$\di_H(q,u)\ne \di_H(q,z')$, which forces $\di_H(q,z')=\di_H(q,v)$. Hence, we must have
$\di_H(q,v)<\di_H(q,u)$. Now
$$\di_G(q,z)=\min\{\di_H(q,u),\di_H(q,v)+l-1\}
>\di_H(q,v)=\di_H(q,z')=\di_G(q,z')$$
implies that $q$ distinguishes $z$ and $z'$, as desired.
Thus, Claim $2$ holds.

Now note that $\mathcal{R}$ is a mixed metric generator of $H$ with the smallest cardinality. Combining Claims $1$ and $2$, we have that $\mathcal{R}\cup \{u,w\}$ is a mixed metric generator of $G$. It follows that $\dm(G)\le \dm(H)+2$.
$\qed$

\begin{lemma}\label{lem-5}
Let $G$ be a $2$-connected graph which is not a cycle.
If every thread of $G$ is short, then $\mathcal{B}$ is a mixed metric generator of $G$.
\end{lemma}
\proof
Take distinct $z,w\in O(G)$. It suffices to show that there exists an element of $\mathcal{B}$ such that it can distinguish $z$ and $w$. We consider three cases as follows.

\medskip
\noindent {\bf Case 1.} $z$ and $w$ lie on the same thread, and $z,w\notin \mathcal{B}$.
\medskip

Take a maximal thread $P=(p_0,p_1,\ldots,p_l)$ of $G$, let $p_0=u$ and $p_l=v$. Then $u,v\in \mathcal{B}$ and
$l<\di_H(u,v)$, where $H=G-P^0$ is connected.
Now assume that $z,w\in O(P)\setminus\{u,v\}$. Note that
$$
\di_G(u,[i,i])=i,~~\di_G(v,[i,i])=l-i, ~~0\le i \le l,
$$
and
$$
\di_G(u,[j-1,j])=j-1,~~\di_G(v,[j-1,j])=l-j, ~~1\le j \le l.
$$
It follows that $z,w$ must be distinguished by some element of $\{u,v\}$, as desired.

\medskip
\noindent {\bf Case 2.} One $z$ and $w$ belongs to $\mathcal{B}$.
\medskip

Without loss of generality, let $w\in \mathcal{B}$.
If $z$ is a vertex or an edge which is not incident with $w$,
then $\di_G(w,z)>0$, which implies that $z$ and $w$ are  distinguished by $w$, as required.
Thus, in the following, we may assume that $z$ is an edge incident with $w$. Note that $\deg(w)\ge 3$.
Starting from $w$, we can obtain a maximal thread $Q$ containing $z$. Let $w'$ be the other endpoint of this maximal thread and let the length of $Q$ be $s$.
Then $w'\in \mathcal{B}$ and $s\le \di_{G-Q^0}(w,w')$. We conclude that
$$
\di_G(w',w)=s,~~\di_G(w',z)=s-1.
$$
As a result, $z$ and $w$ are distinguished by $w'$, as required.

\medskip
\noindent {\bf Case 3.} $z$ and $w$ lie on distinct threads, and $z,w\notin \mathcal{B}$.
\medskip

Let $z$ lies on the maximal thread $P=(p_0,p_1,\ldots,p_l)$ of $G$, and let $p_0=u, p_l=v$.
In the following, we may assume that both $u$ and $v$ do not distinguish $z$ and $w$. Namely,
\begin{equation}\label{eq-7}
\di_G(u,z)=\di_G(u,w),~~\di_G(v,z)=\di_G(v,w).
\end{equation}
As a result, we have that
\begin{equation}\label{eq-8}
\di_G(u,z)+\di_G(v,z)=\left\{
                                  \begin{array}{ll}
                                    l, & \hbox{if
                                    $z$ is a vertex;}\\
                                    l-1, & \hbox{if $z$ is an edge.}
                                  \end{array}
                                \right.
\end{equation}

Suppose, for contradiction, that $w$ is a vertex. Then by \eqref{eq-7} and \eqref{eq-8}, we deduce that $\di_G(u,w)+\di_G(v,w)=l$.
Note that $w\notin O(P)$.
Now Lemma~\ref{lem-1} implies that
$\di_{G-P^0}(u,v)\le l$,
this contradicts that $P$ is short.

We conclude that $w$ is an edge.
Similarly, we can also obtain that $z$ is an edge.
Applying \eqref{eq-7} and \eqref{eq-8} again,
we know that $\di_G(u,w)+\di_G(v,w)=l-1$.
It follows from Lemma~\ref{lem-2} that there exists a path in $G-P^0$ such that its length is at most $l$.
This forces $\di_{G-P^0}(u,v)\le l$,
which contradicts that $P$ is short.
Thus, one of $u$ and $v$ must distinguish $z$ and $w$, as wanted.

Now, combining the above three cases, we conclude that the required result is valid.
$\qed$

\section{Proof of the Main Result}
In this section, we will show Theorem~\ref{other} and confirm Conjecture~\ref{conj}. We begin by citing three known results.

\begin{thm}{\rm (\cite[Corollary 16]{Ch})}\label{leaf}
Let $G$ be a connected graph which is not a tree such that $\ell(G)\ge 1$. Then $\dm(G)\le \ell(G)+2c(G)$.
\end{thm}

\begin{thm}{\rm (\cite[Theorem 4.3]{Kek})}\label{tree}
Let $G$ be a tree. Then $\dm(G)=\ell(G)$.
\end{thm}

\begin{thm}{\rm (\cite[Corollary 12(i)]{Ch})}\label{0-leaf}
Let $G$ be a connected graph with $\ell(G)=0$. If
$G$ has a cut-vertex, then $\dm(G)\le 2c(G)$.
\end{thm}

We are now ready to prove Theorem~\ref{other}.

\medskip

{\noindent \em Proof of Theorem $\ref{other}$.}
Note that $G$ is not a cycle and $\ell(G)=0$. Then
$c(G)\ge 2$.
If every thread of $G$ is short, then by
Lemmas~\ref{lem-3} and \ref{lem-5}, we have that
$$\dm(G)\le |\mathcal{B}|\le 2c(G)-2\le 2c(G),$$ as desired.

In the following, we may assume that
$G$ has a maximal thread $P=(p_0,p_1,\ldots,p_l)$ such that
$\di_{G-P^0}(u,v) \le l$. Now let $H=G-P^0$. Clearly, $H$ is connected and has no leaves. It follows that
$$
c(H)=|E(H)|-|V(H)|+1=(|E(G)|-l)-(|V(G)|-(l-1))+1,
$$
that is,
\begin{equation}\label{meq-1}
c(G)=c(H)+1.
\end{equation}

Suppose that $H$ is a cycle. Then $c(G)=2$. Namely, $G$ is a graph in which $p_0$ and $p_l$ are the only vertices of degree $3$, while every other vertex has degree $2$.
It follows from \cite[Corollaries 8 and 11]{Sed} that
$\dm(G)\le 2c(G)$, as required.

In the following, suppose that $H$ is not a cycle.
Assume that $H$ has a cut-vertex. Theorem~\ref{0-leaf}
implies that $\dm(H)\le 2c(H)$.
Moreover, by Lemma~\ref{lem-4} and \eqref{meq-1}, we deduce that
$\dm(G)\le \dm(H)+2\le 2c(H)+2=2c(G)$,
as wanted.
Thus, next, we may assume that $H$ has no cut vertices. It follows that $H$ is $2$-connected.

Note that $H$ is a $2$-connected graph which is not a cycle. To complete the proof of this theorem, we now proceed by induction on $c(G)$.
If $c(G)=2$, then by \eqref{meq-1}, $H$ is a cycle, and so
the result follows from the above proof.

Now suppose that $c(G)\geq 3$, and assume that the assertion
holds for every graph $G'$ satisfying the hypotheses of the
theorem with $c(G')<c(G)$.
Note that $H$ satisfies the hypotheses of our
theorem. In view of \eqref{meq-1}, by the induction hypothesis, we have $\dm(H)\le 2c(H)$.
Moreover, by Lemma~\ref{lem-4}, we conclude that
$$
\dm(G)\le \dm(H)+2\le 2c(H)+2=2c(G).
$$
This completes the induction and the proof.
$\qed$

Note that a $2$-connected graph has no leaves.
Now, combining Theorems~\ref{leaf}, \ref{tree}, \ref{0-leaf}, and \ref{other},
we conclude that Conjecture~\ref{conj} holds for any connected graph which is not a cycle. Namely, Conjecture~\ref{conj} is valid.

\section*{Conflict of interest}
The authors state no conflict of interest.

\end{document}